\documentclass[reqno]{amsart}
\usepackage[utf8]{inputenc}
\usepackage[T1]{fontenc}
\usepackage{pgf,pgfarrows,pgfnodes,pgfautomata,pgfheaps,pgfshade,hyperref, amssymb}

\usepackage{amssymb}
\usepackage{enumitem}
\usepackage[english]{babel}
\usepackage[capitalize]{cleveref}
\usepackage{mathtools}
\usepackage{soul}

\usepackage{amsmath}
\usepackage{amsthm}

\usepackage{stix}

\hypersetup{
    colorlinks,
    linkcolor={blue!30!black},
    citecolor={green!50!black},
    urlcolor={blue!80!black}
} %colores de las cros-referencias

\usepackage{mathrsfs} % for \SP and \sT 
\usepackage{dsfont}

\newcounter{thmcounter}

\newcounter{introthmcounter}

\usepackage{aliascnt}
\usepackage[capitalize]{cleveref}

\newtheorem{theorem}{Theorem}[section]

\newaliascnt{proposition}{theorem}
\newtheorem{proposition}[proposition]{Proposition}
\aliascntresetthe{proposition}

\newaliascnt{lemma}{theorem}
\newtheorem{lemma}[lemma]{Lemma}
\aliascntresetthe{lemma}

\newaliascnt{corollary}{theorem}

\aliascntresetthe{corollary}

\newaliascnt{definition}{theorem}
\theoremstyle{definition}

\aliascntresetthe{definition}

\theoremstyle{definition}

\newtheorem*{question*}{Question}

\newcounter{proofcount}
\AtBeginEnvironment{proof}{\stepcounter{proofcount}}% count the proofs

\makeatletter                  % reset the claim counter each time proofcount is
\@addtoreset{claim}{proofcount}% increased, effectively this restarts the claims
\makeatother

\theoremstyle{remark}

\newtheorem*{remark*}{Remark}
\newtheorem*{example*}{Example}

\def\R{{\mathbb R}}
\def\Z{{\mathbb Z}}

\def\N{{\mathbb N}}
\def\Q{{\mathbb Q}}

\def\cF{{\mathcal F}}

\def\cB{{\mathcal B}}

\def\cA{{\mathcal A}}

\def\cS{{\mathcal S}}

\newcommand{\1}{\ensuremath{\mathds{1}}}

\title{Partial rigidity and $q$-additive functions}

\author{Trist\'an Radi\'c}
\address{Department of mathematics, Northwestern University, 2033 Sheridan Rd, Evanston, IL, United States of America}
\email{tristan.radic@u.northwestern.edu}

\thanks{The author was partially supported by a Simons
Foundation International grant administered by the Simons Foundation SFI-MPS-SDF-00025709 TR and the
National Science Foundation grant DMS-2348315.}

\makeatletter
\def\@settitle{\begin{center}%
  \baselineskip19\p@\relax
  \normalfont\LARGE % Mantiene la fuente STIX exacta del texto, en tamaño destacado
  \@title
  \end{center}%
}
\makeatother

\begin{document}
%\date{\today}
\maketitle

% \begin{abstract}
% We construct a minimal subshift $(X,S)$ with $2$ ergodic measures $\mu_1, \mu_2$ with different partial rigidity rates $\delta_{\mu_1}, \delta_{\mu_2}$, whose values and partial rigidity sequences are calculated. The system $(X,S)$ has zero entropy, moreover it is a $\cS-$adic subshift of finite alphabetic rank and has non-superlinear word complexity. 

% \end{abstract}

\vspace{-2em}

\begin{abstract}
A system $(X,\mu, T)$ is partially rigid with constant $\delta>0$ if there exists an incresing sequence $(n_k)_{k \in \N}$ such that $\displaystyle \liminf_{k \to \infty } \mu(A \cap T^{-n_k}A) \geq \delta \mu(A)$  for every measurable $A \subset X$. The partial rigidity rate is the largest constant of partial rigidity. We characterize the partial rigidity rate of subshifts arising from $q$-additive functions $\pmod M$ in terms of quantities arising from Bésineau's theorem. With this, we build a family of strongly $q$-additive functions to deduce that the collection of partial rigidity rates for substitution subshifts is dense in $[0,1]$ answering a question of Donoso, Maass and the author. 
\end{abstract}

\section{Introduction}

A measure preserving system $(X, \mu,T)$ is {\em mixing} if for all measurable sets $A \subset X$, $\lim_{n \to \infty} \mu(A \cap T^{-n}A) = \mu(A)^2$. Roughly speaking, in a mixing system, the events from the far future get independent of the initial ones. A uniform constraint to mixing is {\em partial rigidity}, originally introduced in 1989 by Friedman \cite{Friedman_partial_mixing_rigidity_factors:1989} and further developed by King in \cite{King_joining-rank_finite_mixing:1988}. A system is partially rigid if there exists $\delta >0$ and an increasing sequence $(n_k)_{k \in \N}$ such that 
\begin{equation} \label{eq partial rigidity}
    \liminf_{k \to \infty } \mu(A \cap T^{-n_k} A) \geq \delta \mu(A) \quad \text{ for all measurable sets } A\subset X.
\end{equation}
Historically, systems of particular interest that are first proven to be non-mixing end up being partially rigid, see for example  \cite{bruin_karpel_oprocha_radinger2025rigidity,Creutz_mixing_minimal_comp:2023,Danilenko_finite_rank_rationalerg_partial_rigidity:2016,Dekking_Keane_mixing_substitutions:1978,farhangi_Moreira_zelada2026multipliers,Friedman_partial_mixing_rigidity_factors:1989,Katok_interval_exchange_not_mixing:1980}.

The constant of partial rigidity $\delta >0$ in \eqref{eq partial rigidity} captures the degree of dependency of events in the system along the sequence $(n_k)_{k \in \N}$. It is therefore interesting to characterize the largest constant of partial rigidity for a given system $(X,\mu,T)$. That constant is the {\em partial rigidity rate} and it is denoted $\delta_\mu$. Systems with $\delta_\mu =1$ are called {\em rigid} systems and have been extensively studied (see for example \cite{Bergelson_delJunco_Lemanczyk_Rosenblatt_rigidity_nonrecurrence:2014,Coronel_Maass_Shao_seq_entropy_rigid:2009,Donoso_Shao_uniform_rigid_models:2017,Fayad_Kanigowski_rigidity_wm_rotation:2015,Furstenberg_Weiss77}).  Here we are particularly interested in the case $\delta_\mu <1$, where to the best of the author's knowledge, the only results in the literature are \cite{donoso_maass_radic2023partial,elabdalaoui2009spectrumprigid,radic2024multiple_partial_rigidity}.

In this paper we focus on substitution subshifts and more specifically on substitutions arising from $q$-additive functions to answer part of \cite[Question 7.20]{donoso_maass_radic2023partial} (see \cref{thrm density partial rigidity rates}).  In 1978, Dekking and Keane essentially proved that substitution subshifts are partially rigid \cite{Dekking_Keane_mixing_substitutions:1978}.   For a finite alphabet $\cA$, a substitution $\sigma \colon \cA^* \to \cA^*$ is a morphism for word concatenation and, under mild assumptions, $\sigma$ generates a minimal subshift $X_\sigma \subset \cA^\Z$ (see \cref{sec preliminaries} for more details). These systems have restrictive dynamical properties, for instance, they are uniquely ergodic, have zero entropy, sublinear complexity and only continuous eigenfunctions (see \cite{Durand_Perrin_Dimension_groups_dynamical_systems:2022,Queffelec1987} for a detailed exposition). Further restrictions are imposed when the substitution has constant length, for instance, all eigenvalues are rational \cite{dekking1978spectrum_substitutions}.

Abusing notation, we denote by $\delta_\sigma$ the partial rigidity rate of the uniquely ergodic substitution subshift given by $\sigma$. Despite the extensive study of substitution subshifts, little is known about their partial rigidity rates. In this direction, we establish the following result.

\begin{theorem} \label{thrm density partial rigidity rates}
    The set $\Delta_s = \{ \delta_\sigma \colon \sigma \text{ is a primitive constant length substitution}\}$ is dense in $[0,1]$.
\end{theorem}

As mentioned, this was raised in a previous paper joint with Donoso and Maass \cite[Question 7.20]{donoso_maass_radic2023partial} where we established that $1$ is an accumulation point of the set $\Delta_s$, see \cite[Proposition 7.19]{donoso_maass_radic2023partial}. Since there are only countably many substitutions, therefore one cannot expect that $\Delta_s$ equals $(0,1]$. Characterizing the set of partial rigidity rates for other families of dynamical systems (for example, weakly mixing systems, rank-one systems, interval exchange transformations) remains an interesting open question.

One technical difficulty of \cref{thrm density partial rigidity rates} is that, from Danilenko's proof of \cite[Theorem 7.1]{Danilenko_finite_rank_rationalerg_partial_rigidity:2016}, the partial rigidity rate of a substitution subshift on an alphabet with $M$ symbols is greater than or equal to $1/M^M$ (this bound applies in greater generality).  
Therefore to achieve partial rigidity rates that are closer and closer to $0$, one needs to construct substitutions in alphabets with more and more letters. To address that problem we use substitutions arising from strongly $q$-additive functions. 
%This technical difficulty may be solved by brute force using an AI-generated proof, since ideas on how to build these examples were already available in the literature (see for example \cite[Section 7]{donoso_maass_radic2023partial}), but then one would have lost the qualitative information that relates this partial rigidity question with $q$-additive functions and Bésineau's theorem.

In this work $0$ is considered a natural number, so $\N = \{0,1,2,\ldots\}$. For a base $q \geq 2$ of the natural numbers, a  function $f \colon \N \to \N$ is {\em $q$-additive} if $f(0)=0$ and $f(aq^n +b)= f(aq^n)+f(b)$ for all $n \geq1$, $a \in \N$ and $0 \leq b < q^n$. If moreover $f(aq^n) = f(a)$  for all $n \geq1$ and $a \in \N$ then $f$ is a {\em strongly $q$-additive function}.

 A classical example is the sum-of-digits, $s_q(n) = \sum_{j \in \N} \epsilon_j(n)$ (where $n =  \sum_{j \in \N} \epsilon_j(n) q^j$ for $\epsilon_j(n) \in \{0, \ldots,q-1\}$) which is extensively studied in the literature, see for example \cite{drmota_Mauduit_Rivat2019normality_squares,kim1999joint_sum_of_digits,mauduit_rivat2010sum_of_digits_primes,spiegelhofer2022lower_cusick}. 
Concerning general $q$-additive functions, there is great interest in their number theoretic properties and statistical behavior, see for example \cite{coquet1980repartition_q-additives,Delange1972,fan_Konieczny2019uniformity,gelfond1968sum_of_digits,katai2002q-additive,lesigne_mauduiit1996proprietes,lesigne_mauduit_Mosse1994ergodic_q_multiplicative,mauduit_rivat2005proprietes}.

In 1970, Bésineau showed that for any $q$-additive function $f\colon \N \to \N$, any $t \geq 1$ and $m \in \Z$, the set $\{ n \in \N \colon f(n+t) - f(n) = m \}$ has density, that is the limit
\begin{equation} \label{eq besineau result}
    d(\{ n \in \N \colon f(n+t) - f(n) = m \}) = \lim_{N \to \infty } \frac{|\{ n \leq N-1 \colon f(n+t) - f(n) = m \}|}{N}
\end{equation}
exists (see \cite{besineau1970independance}). The original paper only deals with the sum-of-digits function, but the proof extends to this generality. 

We denote the densities in \eqref{eq besineau result} as $\eta_f(t,m) = d(\{ n \in \N \colon f(n+t) - f(n) = m \})$. Notice that for a fixed $t \geq 1$, $(\eta_f(t,m))_{m \in \Z}$ is a probability vector, that is $\sum_{m \in \Z} \eta_f(t,m)=1$ and $\eta_f(t,m) \geq 0$ for all $m \in \Z$. This has been explored previously in the literature (see \cite{Emme_Prikhod2017sum_of_digits,hosten_Janvresse_delaRue2024central_limit_sum_of_digits}).

Before stating our second main result, we need one definition. A subshift $X \subset \cA^\Z$ is {\em generated by a sequence} $(a(n))_{n \in \N} \subset \cA$ if $X$ consists of all the bi-infinite sequences $x \in \cA^\Z$ such that for every index $i \in \Z$ and every $k \in \N$, there is $n \in \N$ such that $(x_i, x_{i+1}, \ldots, x_{i+k}) = (a(n),a(n+1), \ldots, a(n+k)) $. The relation between the constants $\eta_f(t,m)$ from Bésineau's theorem and partial rigidity is the following.

\begin{theorem} \label{thrm partial rigidity main}
    Fix a natural number $M \geq 2$, $\cA = \{0, \ldots, M-1\}$ and a strongly $q$-additive function $f\colon \N \to \N$. If $X_{f,M} \subset \cA^\Z$ is the subshift generated by $(f(n) \mod M)_{n \in \N}$, then $(X_{f,M}, S)$ is uniquely ergodic, with invariant measure $\mu$ and for every $t \geq 1$ the constant
    $
     d(\{ n \in \N \colon f(n+t) = f(n) \pmod M \})$
    is the maximal constant of partial rigidity for the sequence $(t \cdot q^k)_{k \in \N}$. Moreover, the partial rigidity rate is given by
    \begin{equation} \label{eq main partial}
        \delta_\mu = \sup_{t \geq 1} d(\{ n \in \N \colon f(n+t) = f(n) \pmod M \}) . 
    \end{equation}
\end{theorem}

In the previous result one can rewrite $d(\{ n \in \N \colon f(n+t) = f(n) \pmod M \}) $ as $\sum_{n \in \Z} \eta_f(t, Mn)$. The subshift $X_{f,M}$ should be thought of as the orbit closure under the shift map of the sequence $(f(n) \mod M)_{n \in \N}$, with the caveat that we consider bi-infinite sequences so that the shift map is invertible. We also show a version of this theorem for general $q$-additive functions, see \cref{theorem partial rigidity q-additive}. 
The proofs of these theorems crucially use the characterization of partial rigidity rate for $\cS$-adic subshifts from \cite{donoso_maass_radic2023partial}. 

Notice that, when $f=s_2$ and $M =2$, $X_{f,M}$ is the Thue-Morse substitution subshift. Furthermore, all systems $X_{f,M}$ from \cref{thrm partial rigidity main} are substitution subshifts with specific properties described at the end of \cref{sec preliminaries}. In a previous work with Donoso and Maass, this type of substitution arose naturally (see \cite[section 7.2]{donoso_maass_radic2023partial}). They were called ``substitutions of Thue-Morse type'', but we did not recognize the $q$-additive structure hidden in the definition and hence the characterization of the partial rigidity rate in \eqref{eq main partial} stayed unknown. In that same work the partial rigidity rate of the Thue-Morse subshift was computed and, in that case, the supremum in \eqref{eq main partial} is achieved for $t=3$.

Avoiding technicalities, the main idea to deduce \cref{thrm density partial rigidity rates} from \cref{thrm partial rigidity main} is that, for a strongly $q$-additive function and a large $M$, the value $\beta_f = \sup \{ \eta_f(t,0) \colon t \geq 1\}$ approximates the value of $\delta_{\sigma_{f,M}}$  (see \cref{lemma beta_f approx delta_sigma}), where $\sigma_{f,M} $ is the substitution that generates the subshift $X_{f,M}$. The values of $\beta_f$ are easier to compute since we do not have to deal with extra terms arising from the residue classes modulo $M$. We then construct a family $\cF$ of strongly $q$-additive functions so that $\{ \beta_f \colon f \in \cF\}$ is dense in $[0,1]$ (see \cref{prop family of ex}). With this we deduce that $\{ \delta_{\sigma_{f,M}} \colon f \in \cF, M\geq2\}$ is dense in $[0,1]$ and hence conclude \cref{thrm density partial rigidity rates}. 

\textbf{Organization of the paper:} In \cref{sec preliminaries} we introduce notation and classical properties of dynamical systems and symbolic dynamics. We also introduce the substitutions that generate the space $X_{f,M}$. In \cref{sec proof first main}, we prove \cref{thrm partial rigidity main} and its version for general $q$-additive functions. In \cref{sec proof thrm 2}, we develop the argument to show \cref{thrm density partial rigidity rates}. We also include an appendix with the proof of some folklore results relating $q$-additive functions and substitution subshifts. 

\textbf{Acknowledgment:}  I would like to thank Bryna Kra for her feedback on earlier drafts of this paper, and both her and Florian Richter for the conversations that pointed me toward $q$-additive functions. I am also grateful to Andreas Mountakis for his comments and suggestions on previous drafts.

\textbf{Use of technology:} An LLM was used to help with computations of testing by trial and error values of  $f(1), \ldots,f(q-1)$ in \cref{prop family of ex} and for checking proofs and grammar. The mathematical ideas were entirely developed by the author without any use of LLMs.

\section{Preliminaries} \label{sec preliminaries}

A {\em topological dynamical system} is a pair $(X,T)$, where $X$ is a compact metric space and $T \colon X \to X$ is a homeomorphism. The {\em orbit} of a point $x\in X$ is the set $\{T^n x: n\in \Z\}$. We say that $(X,T)$ is {\em minimal} if the orbit of any point is dense in $X$. A $T$-\emph{invariant measure} $\mu$ is a Borel probability measure such that $\mu(T^{-1}A)= \mu(A)$ for all Borel sets $A \subset X$. We say that $(X,T)$ is \emph{uniquely ergodic} if it has only one invariant measure $\mu$. In that case the measure $\mu$ is \emph{ergodic}, that is for every Borel set $A \subset X$ with $\mu(A\Delta T^{-1}A)=0$, we have that $\mu(A)=0$ or $1$.

In this paper a \emph{measure preserving system} (or a {\em system} for simplicity) is a tuple $(X,\mu,T)$ where $(X,T)$ is a topological dynamical system and $\mu$ is a $T$-invariant probability measure. A system is {\em mixing} if for all Borel sets $A \subset X$, $\lim_{n \to \infty} \mu(A \cap T^{-n}A) = \mu(A)^2$ and {\em partially rigid} if there is a constant $\delta >0$ and an increasing sequence $(n_k)$ such that for all Borel sets $A \subset X$,  $\liminf_{k \to \infty} \mu(A \cap T^{-n_k}A) \geq \delta \mu(A)$. If a system is partially rigid then it is not mixing.   The {\em partial rigidity rate} of a system $(X, \mu,T)$ is 
\begin{equation} \label{eq def deltamu}
    \delta_{\mu} = \sup \{ \delta >0 \mid \delta \text{ is a partial rigidity constant of } (X,\mu,T) \text{ for some } (n_k)_{k \in \N} \}
\end{equation}
and by convention $\delta_\mu =0$ if $(X,\mu,T)$ is not partially rigid.

We are interested in symbolic systems, here referred to as subshifts, for a more detailed exposition, refer to \cite{Durand_Perrin_Dimension_groups_dynamical_systems:2022,Queffelec1987}. 

Let $\cA$ be a finite set that we call an {\em alphabet}. The elements in $\cA$ are called {\em letters} or {\em symbols}. For $\ell \in \N$, a {\em word} $w$ of length $\ell$ is a concatenation of $\ell$ symbols $w_0, \ldots , w_{\ell-1} \in \cA$ that is denoted $w = w_0 \cdots w_{\ell-1}$. The length of a word $w$ is denoted by $|w|$ and we only use    ``word'' to refer to finite words. The set of words is denoted $\cA^*$. For a word $w = w_0  \cdots w_{\ell-1}$ and two integers $0 \leq i \leq j < \ell-1$, we write $w_{[i, j+1)} = w_{[i, j]} = w_i  \cdots w_j$. We say that $u$ {\em appears} in $w $ if there is an index $ 0 \leq i \leq |w|-1$ such that $u=w_{[i,i+|u|)}$, which is denoted $u \sqsubseteq w$.

An element $x \in \cA^{\Z}$ is called a bi-infinite sequence. The notation and concepts introduced for finite words are naturally extended for bi-infinite sequences. 
As mentioned in the introduction, the {\em shift map} $S\colon \cA^{\Z} \to \cA^{\Z}$ is defined by $S((x_n)_{n \in \Z})= (x_{n+1})_{n \in \Z}$ for all $x \in \cA^\Z$ and $X \subset \cA^\Z$ is a subshift if it is compact and $S^{-1}(X)=X$. Abusing notation, we denote by $S\colon X \to X$ the shift map restricted to $X$. Notice that the topology of $X$ is generated by the clopen sets $S^{m}[w]_X$ where $m \in \Z$ and $[w]_X $ is the cylinder set $ \{ x \in X \colon x_{[0,|w|) }=w\} $ for $w \in \cA^*$. Abusing notation, for an $S$-invariant measure $\mu$ of $X$, we write $\mu(w)$ to denote $\mu([w]_X)$.

For alphabets $\cA$ and $\cB$, $\sigma\colon \cA^* \to \cB^*$ is a {\em morphism} if it is a morphism for the concatenation, that is $\sigma(uv) = \sigma(u)\sigma(v)$ for all $u,v \in \cA^*$ and when $\cA = \cB$ we say that $\sigma\colon \cA^* \to \cA^*$  is a {\em substitution}. Notice that a morphism is completely determined by the values of $\sigma(a)$ for $a \in \cA$ and for non-degenerate cases we assume $|\sigma(a)| \geq 1$ for all $a \in \cA$.  A substitution is {\em primitive} if there is $n \in \N$ such that for every letter $a,b \in \cA$, $b$ appears in $\sigma^n(a)$. 
For a substitution $\sigma \colon \cA^* \to \cA^*$, the substitution subshift given by $\sigma$ is defined as
\begin{equation*}
    X_\sigma = \{ x \in \cA^\Z \colon x_{[i,i+k)} \sqsubseteq \sigma^n(a) \text{ for every } i \in \Z, k \geq 1 \text{ and some } n \in \N, a \in \cA\}.
\end{equation*}
If the substitution is primitive, then the system $(X_\sigma,S)$ is minimal and uniquely ergodic. 

A \emph{directive sequence} $\boldsymbol \sigma = (\sigma_n\colon \cA_{n+1}^* \to \cA_n^*)_{n \in \N}$ is a sequence of morphisms between finite alphabets $\cA_n$ for all $n \in \N$. 
% In the literature directive sequences are defined more generally allowing one to take morphisms between different alphabets, but this paper does not require that level of generality. 
For $0 \leq n < N$, we denote $\sigma_{[n,N)} = \sigma_{[n,N-1]} = \sigma_n \circ \sigma_{n+1} \circ \cdots \circ \sigma_{N-1}$. 
A directive sequence $\boldsymbol \sigma = (\sigma_n \colon \cA_{n+1}^* \to \cA_n^*)_{n \in \N} $ is \emph{primitive} if for every $n \in \N$ there is $N>n$ such that for every letter $a\in \cA_N$ and $b \in \cA_n$, $b$ appears in $\sigma_{[n,N)}(a)$.  Fixing a primitive directive sequence $\boldsymbol \sigma = (\sigma_n \colon \cA_{n+1}^* \to \cA_n^*)_{n \in \N} $ and $n \in \N$ we denote
\begin{equation*}
    X_{\boldsymbol \sigma}^{(n)} = \{ x \in \cA^\Z_n \colon x_{[i,i+k)} \sqsubseteq \sigma_{[n,N)}(a) \text{ for every } i \in \Z, k \geq 1 \text{ and some } N > n, a \in \cA_{N}\}.
\end{equation*}
For every $n\in \N$, the system $(X^{(n)}_{\boldsymbol \sigma},S)$ is minimal but not necessarily uniquely ergodic. 
$X_{\boldsymbol \sigma} = X_{\boldsymbol \sigma}^{(0)}$ is the $\cS$-\emph{adic subshift} generated by the directive sequence $\boldsymbol \sigma$. Notice that a substitution subshift given by $\sigma$ is an $\cS$-adic subshift with directive sequence constantly equal to $\sigma$. 

% For a $\boldsymbol \sigma$ primitive and recognizable directive sequence (see definition in \cite{Berthe_Steiner_Thuswaldner_Recognizability_morphism:2019}), if $(X_{\boldsymbol \sigma},S)$ is uniquely ergodic, then $(X_{\boldsymbol \sigma}^{(n)},S)$ is uniquely ergodic for all $n \in \N$ (see \cite[Theorem 6.5]{Berthe_Steiner_Thuswaldner_Recognizability_morphism:2019} and \cite{Bezuglyi&Kwiatkowski&Medynets&Solomyak:2013}). For this paper we only use that if, for every  $a \in \cA$ and $n \in \N$, the first letter of $\sigma_n(a)$ is $a$ then $\boldsymbol{\sigma}$ is recognizable.  

A substitution has {\emph constant length} if there is $\ell \geq 2$ such that $|\sigma(a)| =\ell$ for all $a \in \cA$. A directive sequence $\boldsymbol \sigma = (\sigma_n\colon \cA^* \to \cA^*)_{n \in \N}$ has {\emph constant length} if every substitution $\sigma_n$ has constant length for all $n \in \N$. 
% A substitution $\sigma$ is {\em aperiodic} if $X_\sigma$ is not finite and analogously for a directive sequence.

A $q$-\emph{additive function} is a function so that $f(0)=0$ and $f(aq^\ell + b) = f(a q^\ell) + f(b)$ for all $0\leq b < q^\ell$ and $a, \ell \in \N$. If, in addition, $f(aq^\ell) =f(a)$ for all $a,\ell \in \N$ then $f$ is \emph{strongly $q$-additive}. Notice that a strongly $q$-additive function is completely determined by its values $f(i)$ for $i \in \{1, \ldots, q-1\}$. 

For a $q$-\emph{additive function}, $\alpha \in \R$ and denoting $e(t) = \exp(2\pi i t)$, we say that $n \mapsto e(f(n) \alpha)$ is a $q$-multiplicative function. Since in this work we essentially consider $\alpha = 1/M$ for an integer $M \geq 2$, we do not use the complex number notation and use the $\pmod M$ one. 
If we fix the $q$-multiplicative function $u= (f(n) \mod M)_{n \in \N}$ and we fix $\cA = u(\N)$ which is a subset of $\Z/M\Z$ that we identify with a subset of $\{0, \ldots, M-1\}$, we denote by $X_{f,M}$ generated by the sequence $(u(n))_{n \in \N}$, that is 
\begin{equation*}
    X_{f,M} = \{ x \in \cA^\Z \colon \text{ for all } i \in \Z, k \in \N \text{ there is } n \in \N \text{ with } x_{[i,i+k]} = u(n)\cdots u(n+k)\}.
\end{equation*}
By \cite[Theorem 3.9]{lesigne_mauduit_Mosse1994ergodic_q_multiplicative}, for any $q$-additive function and $M \geq 2$, $(X_{f,M},S)$ is minimal and uniquely ergodic. 

We now give an $\cS$-adic representation of $(X_{f,M},S)$. First, for $f$ strongly $q$-additive, $M \geq 2$ and as before take $\cA = u(\N)$ for $u = (f(n) \mod M)_{n \in \N}$ which is a subset of $\Z/M\Z$. Moreover, in this case, $\cA$ is a subgroup of $\Z/M\Z$. Define $\sigma_{f,M} \colon \cA^* \to \cA^*$ as
\begin{equation*}
    \sigma_{f,M}(a) = a(a+f(1) \mod M)(a+f(2) \mod M)\cdots(a+f(q-1) \mod M)
\end{equation*}
for all $a \in \cA$. Notice that $\sigma_{f,M}$ has constant length $q$.
The substitution $\sigma_{f,M}$ is primitive and, from \cref{prop subs strongly q-additive}, $X_{\sigma_{f,M}}=X_{f,M}$ (see \cref{prop q-additive mod M directive sequence}). We include \cref{prop subs strongly q-additive} in the appendix since it is considered folklore and the author did not find a proof for it. Notice that $\tau=\sigma_{s_2, 2}$ is the Thue-Morse substitution (ie $\tau(0) = 01$, $\tau(1)=10$). 

For the general $q$-additive case, the construction is similar but using directive sequences. If $f$ is $q$-additive, for every $n \in \N$, define the $q$-multiplicative function $u_n$ given by $u_n(m) = f(mq^n) \pmod M$ for $m \in \N$ and take $\cA_n = u_n(\N)$.  
Notice that $\cA_n \subset \{0, \ldots, M-1\}$, that $0 \in \cA_n$ and $\cA_{n+1} \subset \cA_{n}$ for all $n \in \N$. Define similarly the directive sequence $\boldsymbol{\sigma}=(\sigma_{n,f,M} \colon \cA_{n+1}^* \to \cA_n^*)_{n \in \N}$ by 
\begin{equation} \label{eq directive sequence}
    \sigma_{n,f,M}(a) = a(a+f(q^n) \mod M)\cdots(a+f((q-1)q^n) \mod M)
\end{equation}
for all $a \in \cA_{n+1}$.
\begin{proposition} \label{prop q-additive mod M directive sequence}
    For a $q$-additive function $f\colon \N \to \N$ and $M \geq 2$ the directive sequence $\boldsymbol{\sigma}=(\sigma_{n,f,M} \colon \cA_{n+1}^* \to \cA_n^*)_{n \in \N}$ defined in \eqref{eq directive sequence} is well defined and primitive. Moreover, $X_{\boldsymbol \sigma} = X_{f,M}$.
\end{proposition}

\begin{proof}
    If we define $\boldsymbol{\tau}=(\tau_{n} \colon \{0,\ldots,M-1\}^* \to \{0,\ldots,M-1\}^*)_{n \in \N}$ by
    \begin{equation*} 
    \tau_{n}(a) = a(a+f(q^n) \mod M)\cdots(a+f((q-1)q^n) \mod M)
\end{equation*}
    for all $a \in \{0, \ldots,M-1\}$ and $u_n(m) = f(mq^n) \pmod M$ for $m \in \N$, then by \cref{prop subs general q-additive}, $u_n$ equals $\lim_{N \to \infty} \tau_{[n,N)}(0)$ for all $n \in \N$. Notice that the directive sequence $\boldsymbol\tau$ basically equals $\boldsymbol{\sigma}$ but in the latter at each level $n \in \N$ one restricts the alphabet to $\cA_n=u_n(\N)$. Then notice that $u_n=\tau_{n}(u_{n+1}) = \sigma_{n,f,M}(u_{n+1})$, so $\boldsymbol{\sigma}$ is well defined.    
    
    Thus,  $u_n$ equals $\lim_{N \to \infty} \sigma_{[n,N)}(0)$ (where we are using $\sigma_n = \sigma_{n,f,M}$). In particular,  for every $n\in \N$, there is $N >n$ such that for all $b \in \cA_n$, $b$ appears in $\sigma_{[n,N)}(0)$ (we are using again that $\cA_n =u_n(\N)$). By construction for all $a \in \cA_N$, $\sigma_{[n,N)}(a) = \sigma_{[n,N)}(0) \oplus a \pmod M$, where here for a word $w =w_0w_1 \cdots w_{\ell-1}$ and a letter $a$, $w\oplus a \pmod M = (w_0 + a \mod M)(w_1 + a \mod M)\cdots (w_{\ell-1} + a \mod M)$. Therefore every letter of $\cB = \{ c+a \pmod M\colon c \in \cA_n\}$ appears in $\sigma_{[n,N)}(a)$. Since $|\cB|=|\cA_n|$ and $\boldsymbol{\sigma}$ is well defined, we get that $\cB = \cA_n$ and therefore, for any $a\in \cA_N$ and any $b \in \cA_n$, $b$ appears in $\sigma_{[n,N)}(a)$, concluding the primitivity. 

    The fact that $X_{f,M}= X_{\boldsymbol{\sigma}}$ is direct from the easy contention that $X_{f,M} \subset X_{\boldsymbol \sigma}$ and the minimality of $(X_{\boldsymbol \sigma},S)$ ensured by the primitivity of the directive sequence. 
\end{proof}

We notice that restricting the alphabet is important for primitivity. For example if $f = 2s_2$ (the double of the sum-of-digits in binary), then $\sigma_{f,4} \colon \{0,2\}^* \to \{0,2\}^*$ is given by $0 \mapsto 02$ and $2 \mapsto 20$ and the substitution subshift it induces is a copy of the Thue-Morse subshift but in the alphabet $\{0,2\}$. In the naive version, where the alphabet is not restricted, $\tau \colon \{0,1,2,3\}^* \to \{0,1,2,3\}^*$ given by $0 \mapsto 02$, $1 \mapsto13$, $2 \mapsto 20$ and $3 \mapsto31$, the subshift $X_{\tau}$ is the disjoint union of two copies of the Thue-Morse substitution subshift but in the alphabets $\{0,2\}$ and $\{1,3\}$.

Notice that \cref{prop q-additive mod M directive sequence} is another proof of the fact that $X_{f,M}$ is minimal. However it is not direct from it that $X_{f,M}$ is uniquely ergodic so we rely on \cite[Theorem 3.9]{lesigne_mauduit_Mosse1994ergodic_q_multiplicative} for that fact.
Finally, we notice that some of these systems are periodic which are not interesting for studying their partial rigidity.
For example, suppose that $f\colon \N \to \N$ is a $q$-additive function for which there is $n_0 \in \N$ such that $f(iq^n)=0$ for all $n \geq n_0$ and $i\in \{1, \ldots,q-1\}$. In this case already $f$ is periodic, and notice that we end up in a degenerate case where for $\sigma_{n,f,M} \colon \cA_{n+1}^* \to \cA_n^*$ from \eqref{eq directive sequence}, $\cA_n = \{0\}$ for all $n \geq n_0$. More interestingly, if $f = s_3$, $\sigma_{s_3, 2}$ is given by $0 \mapsto010$ and $1 \mapsto 101$ which induces a two point system. 
% Since we are interested in partial rigidity, we also state the theorems under the assumption that $\sigma_{f,M}$ is aperiodic, even if they are also true in the periodic case. 

In what follows we assume that the systems are aperiodic, even though this assumption is not strictly necessary. We finish this section by noticing that, $\tau = \sigma_{f,M}$ is {\em one-to-one in letters} (ie the $i$th letter of $\tau(a)$ is different to the $i$th letter of $\tau(b)$ for $a \neq b$ and all $i \in \{0,1,\ldots,q-1\}$), so if we assume that $\tau$ is aperiodic \cite[Theorem 2]{lemanczyk_mentzen1988metric} together with \cite[Theorem 7]{dekking1978spectrum_substitutions} ensure that $\delta_\tau < 1$ (see also \cite[Proposition 7.11]{donoso_maass_radic2023partial}). 

\section{Partial rigidity rate for subshift generated by $q$-additive functions} \label{sec proof first main}

To compute the partial rigidity rate we use the following special case of \cite[Theorem 7.1]{donoso_maass_radic2023partial}.
Similarly to the introduction, if a directive sequence induces a uniquely ergodic $\cS$-adic subshift $(X_{\boldsymbol{\sigma}}, \mu,S)$ then we denote by $ \delta_{\boldsymbol \sigma}$ the partial rigidity rate. 

\begin{lemma} \label{lemma partial rigidity complete words}
    Let $\boldsymbol \sigma = (\sigma_n \colon \cA^* \to \cA^*)_{n \in \N}$ be an aperiodic primitive and recognizable and constant length directive sequence with $X_{\boldsymbol \sigma}$ uniquely ergodic. Denote $\mu^{(n)}$ the unique invariant measure at each level $n \in \N$. Then
    \begin{equation*}
        \delta_{\boldsymbol \sigma} = \lim_{n \to \infty} \sup_{t \geq 1} \bigg( \sum_{w \in \cA^* \colon |w|=t+1, w_0 = w_t} \mu^{(n)}(w) \bigg). 
    \end{equation*}
\end{lemma}

The words $w=w_0\cdots w_t$ that start and finish with the same letter are called {\em complete words}. \cref{lemma partial rigidity complete words} asserts that complete words control partial rigidity. Recall that for a $q$-additive function we associate a directive sequence $(\sigma_{f_n,M} \colon \cA^* \to \cA^*)_{n \in \N}$ (see \cref{sec preliminaries}). In this paper we do not define the notion of recognizability but we use that if the directive sequence $\boldsymbol \sigma =(\sigma_n \colon \cA_{n+1}^* \to \cA^*_n)_{n \in \N}$ (with $\cA_{n+1} \subset \cA_n$ $n \in \N$) is aperiodic and satisfies that the first letter of the word $\sigma_n(a)$ is $a$, for all $a \in \cA_{n+1}$ and $n \in \N$, then $\boldsymbol{\sigma}$ is recognizable, see \cite[Theorem 3.1]{Berthe_Steiner_Thuswaldner_Recognizability_morphism:2019}.

\begin{theorem} \label{theorem partial rigidity q-additive}
    Let $f \colon \N \to \N$ be a $q$-additive sequence, $M \geq 2$ and $(X_{f,M}, \mu,S)$ the minimal and uniquely ergodic subshift generated by $(f(n) \mod M)_{n \in \N}$. Then 
    \begin{equation} \label{eq partial rigid por general q-additive}
        \delta_{\mu} = \lim_{n \to \infty} \sup_{t \geq 1} d(\{m \in \N \colon f(q^n(m+t)) =f(q^nm) \pmod M \}). 
    \end{equation}
\end{theorem}

\begin{proof} Take the directive sequence $\boldsymbol{\tau} = (\sigma_{n,f,M} \colon \cA_{n+1}^* \to \cA_n^*)_{n \in \N}$ from \eqref{eq directive sequence} and for $n \in \N$, the sequence $m \mapsto u_n(m) = f(q^n m) \pmod M$. Using \cref{prop q-additive mod M directive sequence} we get that $X_{\boldsymbol \tau}^{(n)}$ is the subshift generated by the sequence $(u_n(m))_{m \in \N}$ and denote $\mu^{(n)}$ its unique invariant measure.

    For a fixed $n \in \N$, consider $x \in X_{\boldsymbol \tau}^{(n)}$ such that $x_i = u_n(i)$ for all $i \in\N$. Thus, for every  $w =w_0w_1 \cdots w_t \in \cA^*_n$ 
    \begin{align*}
        \mu^{(n)}(w) &= \lim_{N \to \infty} \frac{1}{N} \sum_{j=0}^{N-1} \1_{ [w]_{X^{(n)}_{\boldsymbol{\tau}}}}(S^jx ) = d(\{m \in \N \colon u_n(m)=w_0, \ldots, u_n(m+t) = w_t \}).
    \end{align*}
     Finally, summing the previous expression over $\{w \in \cA^*_n \colon |w|=t+1, w_0 = w_t\}$ we get
    \begin{align*}
       &\sum_{w \in \cA^*_n \colon |w|=t+1, w_0 = w_t} d(\{m \in \N \colon u_n(m +i)=w_i ,\text{ for } i=0,\ldots,t\}) \\
       =&\sum_{w \in \cA^*_n \colon |w|=t+1, w_0 = w_t} d(\{m \in \N \colon f(q^n(m +i))=w_i \pmod M ,\text{ for } i=0,\ldots,t\}) \\
       =&d(\{m \in \N \colon f(q^n(m +t))=f(q^nm) \pmod M\}).
    \end{align*}
    And then we conclude with \cref{lemma partial rigidity complete words}.
\end{proof}

 We finish this section by proving \cref{thrm partial rigidity main}. For this we state the following lemma.

\begin{lemma} \label{lemma direct from Theorem A}
    Let $\sigma \colon \cA^* \to \cA^*$ be a primitive one-to-one in letters substitution with constant length $q$ and let $(X_\sigma,\mu,S)$ be the substitution subshift generated by $\sigma$. Then for every $t \geq 1$
    \begin{equation*}
         \eta_\sigma(t) =\sum_{w \in \cA^* \colon |w|=t+1, w_0 = w_t} \mu(w)  
    \end{equation*}
    is a constant of partial rigidity for the sequence $(t \cdot q^n)_{n \in \N}$. Moreover for that sequence $\eta_\sigma(t)$ is the maximal constant of partial rigidity.  
\end{lemma}

\begin{proof}
    The fact that $\eta_\sigma(t)$ is a constant of partial rigidity for $(t \cdot q^n)_{n \in \N}$ is a direct consequence of \cite[Theorem A and Remark 3.6]{donoso_maass_radic2023partial}. To see that it is maximal for that sequence, we first show that for any $\ell \geq 1$
    \begin{equation} \label{eq qell +1 = ell +1}
        \sum_{u \in \cA^* \colon |u|=\ell+1, u_0 = u_{\ell}} \mu(u) = \sum_{w \in \cA^* \colon |w|=q\ell+1, w_0 = w_{q\ell}} \mu(w).
    \end{equation}
    Indeed, fix $u \in \cA^*$ of length $\ell+1$, if $v = \sigma(u)$ and $w^{(i)}= v_{[i,\ell q+i]}$ for all $i \in \{0,\ldots, q-1\}$, then $w^{(i)}$ is a complete word for all $i \in \{0, \ldots,q-1\}$ if and only if $u$ is a complete word. This is direct from the fact that the substitution is one-to-one in letters, since $u_0 = u_\ell$ if and only if $w^{(i)}_0=[\sigma(u_0)]_i = [\sigma(u_\ell)]_i=w^{(i)}_{q\ell}$ for all $i \in \{0, \ldots,q-1\}$. 
    This together with \cite[Theorem 3]{Frid_frequency_D0L:1998} establishes \eqref{eq qell +1 = ell +1}. 
    
    Using \eqref{eq qell +1 = ell +1} inductively, we get that for all $t \geq 1$,
    \begin{equation} \label{eq q^nt}
        \eta_\sigma(t) =\sum_{u \in \cA^* \colon |u|=t+1, u_0 = u_{t}} \mu(u) = \sum_{w \in \cA^* \colon |w|=q^n t+1, w_0 = w_{q^nt}} \mu(w).
    \end{equation}
    Now, suppose for a contradiction that there are constants $\delta'>\delta>\eta_\sigma(t)$ such that for any measurable set $A \subset X_\sigma$, $\liminf_{n \to \infty} \mu(A \cap S^{t \cdot q^n}A) \geq \delta' \mu(A)$. In particular there is $n \in \N$ sufficiently large so that for all $a \in \cA$, $\mu([a]_X \cap S^{t \cdot q^n}[a]_X) \geq \delta \mu(a)$.
    Notice that 
    \begin{equation*}
        \mu([a]_X \cap S^{t \cdot q^n}[a]_X) = \sum_{w \in \cA^* \colon |w|=q^nt+1, w_0 = w_{q^nt}=a} \mu(w).
    \end{equation*}
    Therefore,
    \begin{equation*}
        \sum_{a\in \cA} \mu([a]_X \cap S^{t \cdot q^n}[a]_X) =  \sum_{w \in \cA^* \colon |w|=t\cdot q^n+1, w_0 = w_{q^nt}} \mu(w)  = \eta_\sigma(t) \geq \sum_{a\in \cA} \delta \mu(a) = \delta, 
    \end{equation*}
    which is a contradiction. 
\end{proof} 

\begin{proof}[Proof of \cref{thrm partial rigidity main}]
    The first part is direct from \cref{lemma direct from Theorem A} for $\sigma_{f,M}$ together with the fact that, similarly to the proof of \cref{theorem partial rigidity q-additive},
    \begin{equation*}
        \eta_{\sigma_{f,M}} (t) = d(\{ n \in \N \colon f(n+t) =f(n) \pmod M  \}).
    \end{equation*}
    Using that a strongly $q$-additive sequence $f$ is in particular a $q$-additive sequence, from \cref{theorem partial rigidity q-additive} and using that $f(q^nm) = f(m)$, there is no dependency on $n$ for \eqref{eq partial rigid por general q-additive} so one directly deduces that 
    \begin{equation*} 
        \delta_{\sigma_{f,M}} =  \sup_{t \geq 1} d(\{m \in \N \colon f(m+t) =f(m) \pmod M\}). \qedhere
    \end{equation*}
\end{proof}

\section{Proof of \cref{thrm density partial rigidity rates}} \label{sec proof thrm 2}

In this section we only consider functions $f$ that are strongly $q$-additive. For $t \geq 1$ and $m \in \Z$, we use the following notation
\begin{align*}
    \eta_f(t,m) &=  d( \{ n \in \N \colon f(n+t) -f(n) = m \})\\
    \alpha_f &= \sup \{ d( \{ n \in \N \colon f(n+t) - f(n) = m \}) \ \colon \ t \geq 1, m \in \Z\}\\
    \beta_f &= \sup\{ d( \{ n \in \N \colon f(n+t) = f(n) \}) \ \colon  \ t \geq 1\}. 
\end{align*}
In other words $\alpha_f = \sup \{ \eta_f(t,m) \colon t \geq 1, m \in \Z\}$ and $\beta_f = \sup\{ \eta_f(t,0) \colon t \geq 1\}$. As mentioned in the introduction, $\eta_f(t,m)$ is well defined due to Bésineau's theorem \cite{besineau1970independance}. Before proving \cref{thrm density partial rigidity rates}, we construct a family of strongly $q$-additive functions where we can control the value of $\beta_f$ (see \cref{prop family of ex}). We need the following lemma,

\begin{lemma} \label{lemma max beta_f}
    Let $f$ be a strongly $q$-additive function, then 
    \begin{equation} \label{eq max beta_f}
        \alpha_f =  \sup \{ \eta_f (s,n) \colon s \in \{ 1, \ldots, q-1\}, n\in \Z \}.
    \end{equation}
     Moreover, the supremum in the previous equation is a maximum as well as the supremum that defines $\alpha_f$. In particular, $\beta_f \leq \max \{ \eta_f (s,n) \colon s \in \{ 1, \ldots, q-1\}, n\in \Z \}$. 
\end{lemma}

\begin{proof}
    There is one inequality in \eqref{eq max beta_f} that is clear. To prove the converse inequality we actually prove that for every $m \in \Z$, $t\geq 1$, $\eta_f(t,m)$ is the convex combination of elements of $\{ \eta_f (s,n) \colon s \in \{ 1, \ldots, q-1\}, n\in \Z \}$. This fact is direct from applying inductively the following identity, for $0 \leq i < q$ and $s\geq 1$ consider $t= qs+i$ 
    \begin{align*}
        \eta_f(qs+i, m) = d( \{ n \in \N \colon f(n+qs+i) - f(n) = m \})& \\
        = \frac{1}{q} \sum_{j=0}^{q-1} d( \{ n \in \N \colon f(qn+qs+i+j) - f(qn+j) = m \})& \\
        = \frac{1}{q} \bigg(\sum_{j=0}^{q-1-i} \eta_f ( s,  f(j) -f(i+j) + m ) + \sum_{j=q-i}^{q-1} \eta_f(s+1, f(j) - f(i+j-q) + m)\bigg)&
    \end{align*}
    where the last equality is a consequence of the strong $q$-additivity of $f$. 
    
    To see that the supremum in the right hand side of \eqref{eq max beta_f} is a maximum, fix $s \in \{ 1, \ldots, q-1\}$ and notice that $(\eta_f(s,n))_{n \in \Z}$ is a probability vector, hence $\sup \{ \eta_f (s,n) \colon  n\in \Z \} = \max \{ \eta_f (s,n) \colon  n\in \Z \} $. Since $\{ 1, \ldots, q-1\}$ is a finite set, we conclude that the right-hand side of \eqref{eq max beta_f} is attained. By the previous part, the same holds for the left-hand side.
\end{proof}

\begin{proposition}  \label{prop family of ex}
    Let $q$ and $k$ be integers such that $q \geq 5$ and $2 \leq k \leq q-3$.  Let $f$ be the strongly $q$-additive function given by 
    \begin{equation}
        f(n) = \begin{cases}
            \begin{array}{cc}
                 0 & \text{ for } 1 \leq n \leq k  \\
                2^{n-k}-1 & \text{ for } k < n \leq q-1 ,
            \end{array}
        \end{cases}
    \end{equation}
    then $\beta_f = k/q$. Moreover, $\beta_f = \alpha_f = \eta_f(1,0)$.
\end{proposition}

\begin{proof}
    Fix $k$, $q$ and $f$ as in the statement. We use the notation $n = \sum_{j \in \N} \epsilon_j(n) q^j$ for $\epsilon_j(n) \in \{0, \ldots, q-1\}$.
    
    \underline{Case $t=1$:} Let $\ell=\ell(n)$ be defined as $\ell(n) =0$ if $\epsilon_0(n) \in \{0,\ldots,q-2\}$ and if $\epsilon_0(n)=q-1$ it is the minimal index so that $\epsilon_{\ell-1}(n)=q-1$ but $\epsilon_\ell(n) \neq q-1$. With this,
    \begin{align*}
        f(n+1) - f(n) &= f(\epsilon_\ell(n) + 1) - f(\epsilon_\ell(n)) - \ell f(q-1) \\
        &= 2^{\epsilon_\ell(n)-k} \cdot \1_{\epsilon_\ell(n) \geq k}- \ell \cdot (2^{q-1-k}-1)
    \end{align*}
    Notice that when $\ell(n) \geq 1$, that is when $\epsilon_0(n) = q-1$, then $f(n+1)-f(n)$ is negative (since for all $n \in \N$ $2^{\varepsilon_\ell (n)-k} \leq 2^{q-1-k}-1$). When $\epsilon_0(n) \in \{ 0, \ldots, k-1\}$, or equivalently when $n \in \bigcup_{j=0}^{k-1} q \N +j$, then $f(n+1)-f(n) = 0$. Finally, when $\epsilon_0(n) = j$ for $k \leq j < q-1$,  or equivalently when $n \in  q \N +j$, we get that $f(n+1)-f(n) = 2^{j-k}$.  Thus, $\eta_f(1,0) = k/q$, for $k \leq j < q-1$, $\eta_f(1,2^{j-k}) = 1/q$ and for all the other $m \in \Z$, $\eta_f(1,m) \leq 1/k$. 

    \underline{Case $2 \leq t \leq q-1$:} Notice that if $\epsilon_0(n)  \leq q-1-t$ we get that $f(n+t)-f(n) = f(\epsilon_0(n)+t) - f(\epsilon_0(n))$. Otherwise, $f(n+t)-f(n) $ equals 
    \begin{equation*}
        f(\epsilon_0(n)+t-q) - f(\epsilon_0(n)) - (\ell -1) f(q-1) + f(\varepsilon_\ell(n)+1)-f(\varepsilon_\ell(n))
    \end{equation*}
    where in this case for $\epsilon_0(n)  \geq q-t$, $\ell = \ell(n)$ is the minimal index (greater than or equal to $1$) so that $\epsilon_\ell(n) \neq q-1$. Notice that if $n'= \sum_{j\in \N} \epsilon_{j+1}(n)q^j$ (that is $\epsilon_j(n')= \epsilon_{j+1}(n)$ for all $j \in \N$), then we can write 
    \begin{equation*}
        f(n+t)-f(n) = \begin{cases}
            \begin{array}{cc}
                 f(\epsilon_0(n)+t) - f(\epsilon_0(n)) & \textit{ if } \epsilon_0(n)  \leq q-1-t \\ 
                f(\epsilon_0(n)+t-q) - f(\epsilon_0(n)) + f(n'+1)-f(n') & \textit{ if } \epsilon_0(n)  \geq q-t  .
            \end{array}
        \end{cases}
    \end{equation*}
    With this
    \begin{equation} \label{eq eta_f recursive formula}
        \eta_f(t,m) = \frac{1}{q} \sum_{i=0}^{q-1-t} \1_{f(i+t) - f(i)=m} + \frac{1}{q} \sum_{i=q-t}^{q-1}  \eta_f(1,m+f(i) -f(i+t-q))
    \end{equation}
    First, if $m<0$, there is no $i\in \{0, \ldots,q-1-t\}$ for which $ f(i+t) - f(i)=m$, so bounding $ \eta_f(1,m+f(i) -f(i+t-q)) \leq k/q$, we get that $\eta_f(t,m) \leq tk/q^2 \leq k/q$. Now, if $m >0$,
    there is at most one $i\in \{0, \ldots,q-1-t\}$ for which $ f(i+t) - f(i)=m$. However, notice that since $i+t-q < i$, by construction $f(i)-f(i+t-q)$ is non-negative for all $i\in \{q-t, \ldots,q-1 \}$. In particular, since $m>0$, $f(i)-f(i+t-q) \neq -m $, hence $\eta_f(1,m+f(i) -f(i+t-q)) \leq 1/q$, for all $i\in \{q-t, \ldots,q-1 \}$. Thus, $\eta_f(t,m) \leq 1/q + t/q^2 \leq 2/q$, for all $m > 0$. 

    For $m = 0$, we need to bound \eqref{eq eta_f recursive formula}. To do so, notice that $\sum_{i=0}^{q-1-t} \1_{f(i+t) = f(i)} = \max\{0,k-t+1\}$. For the second part of the expression, 
    \begin{align*}
        &\frac{1}{q}\sum_{i=q-t}^{q-1}  \eta_f(1,f(i) -f(i+t-q)) = \frac{1}{q}\sum_{i=0}^{t-1}  \eta_f(1,f(i+(q-t)) -f(i)) \\
        = &\frac{1}{q}\max(0,k-(q-t)+1) \cdot \eta_f(1,0) + \frac{1}{q}\sum_{i=\max(0,k-(q-t)+1)}^{t-1}  \eta_f(1,f(i+(q-t)) -f(i)) \\
        \leq&  \frac{\max(0,k-(q-t)+1) k+ (t-\max(0,k-(q-t)+1))}{q^2}.
    \end{align*}
    Therefore, when $\max(0,k+1-(q-t))  =0$ the previous expression is bounded by $t/q^2 \leq 1/q$. 
    Otherwise, if $k+1\leq  t$ 
     \begin{equation*}
         \eta_f(t,0) \leq\frac{(k-(q-t)+1)(k-1)+ t}{q^2} \leq \frac{k-1}{q} + \frac{1}{q} = \frac{k}{q}.
     \end{equation*}
     and if $k+1> t$, notice that $k - (q-t) +1 \leq q-2 -(q-t)= t-2$, so the expression is bounded by $((t-2) k+t)/q^2 \leq ((t-2)(k+1) +2 )/q^2$ and therefore
    \begin{equation*}
        \eta_f(t,0) \leq \frac{k-t+1}{q} + \frac{(t-2)(k+1) +2 }{q^2} \leq \frac{k-t+1}{q} + \frac{t-1 }{q}  = \frac{k}{q}.
    \end{equation*}
    Thus, for any $t\in \{1, \ldots, q-1\}$, $\eta_f(t,0) \leq \frac{k-t+1}{q} + \frac{t-1}{q} = \frac{k}{q} $. 

    Summarizing, we have that $\eta_f(1,0) = k/q$ and that for all $m \in \Z$ and $t \in \{1, \ldots, q-1\}$, $\eta_f(t,m) \leq k/q $ so we conclude the proposition with \cref{lemma max beta_f}. 
\end{proof}

To transition to substitution subshifts we use the following lemmas.

\begin{lemma} \label{lemma alpha_f mod M}
    Let $f$ be a strongly $q$-additive function, then 
    \begin{equation*}
        \alpha_f = \lim_{M \to \infty} \left( \sup_{t \geq 1, m \in \Z} d( \{ n \in \N \colon f(n+t) - f(n) =m \pmod M \}) \right) .
    \end{equation*}
\end{lemma}

\begin{proof}
    Fix a $q$-additive function $f$. Abusing notation, we write $\eta_f(t,m \mod M) = d( \{ n \in \N \colon f(n+t) - f(n) =m \pmod M \})$, for $M \geq 2$. 
    Using the same steps of the proof of \cref{lemma max beta_f}, we get that for any $t \geq 1,$ $m \in \Z$ $\eta_f(t,m \mod M)$ is a convex combination of the values in $\{ \eta_f(s,n \mod M) \colon s \in \{1, \ldots, q-1\}, n \in \{-\lfloor\frac{M-1}{2}\rfloor, \ldots, \lfloor\frac{M}{2}\rfloor\}\}$. Therefore, if we denote
    $\alpha^M_f = \sup \{ \eta_f(t,m \mod M) \colon t \geq 1, m \in \Z  \}$, then 
    \begin{equation*}
        \alpha^M_f = \max \bigg\{ \eta_f(t,m \mod M) \colon t \in \{1, \ldots, q-1\}, m \in \Big\{-\lfloor\frac{M-1}{2}\rfloor, \ldots, \lfloor\frac{M}{2}\rfloor\Big\}\bigg\}.
    \end{equation*}
    Notice that  $\alpha_f^M \geq \alpha_f$ for every $M \geq 2$, so we only need to prove that $\limsup_{M \to \infty} \alpha_f^M = \alpha_f$.
    \begin{align*}
        \limsup_{M \to \infty} \alpha_f^M =& \limsup_{M \to \infty} \max_{t \in \{1, \ldots, q-1\}} \Big(\max_{ -M/2 < m \leq M/2} \eta_f(t,m \mod M) \Big) \\
        =&\max_{t \in \{1, \ldots, q-1\}} \limsup_{M \to \infty}  \Big(\max_{ -M/2 < m \leq M/2}\sum_{n \in \Z} \eta_f(t,m + Mn) \Big) \\
        =& \max_{t \in \{1, \ldots, q-1\}} \limsup_{M \to \infty}  \Big(\max_{ -M/2 < m \leq M/2} \eta_f(t,m) \Big)\\ &+ \max_{t \in \{1, \ldots, q-1\}} \limsup_{M \to \infty}  \Big(\max_{ -M/2 < m \leq M/2} \sum_{n \in \Z \setminus \{0\}} \eta_f(t,m + Mn) \Big) 
    \end{align*}
    where the first term of the last equality is just $\alpha_f$ by \cref{lemma max beta_f} and the second one is non-negative and bounded above by
    \begin{equation*}
         \limsup_{M \to \infty} \max_{t \in \{1, \ldots, q-1\}} \bigg(1- \sum_{-M/2 < n \leq M/2} \eta_f(t,n) \bigg)
    \end{equation*}
    which is zero because for every $t \in \{1, \ldots, q-1\}$, $(\eta_f(t,n))_{n \in \Z}$ is a probability vector, concluding the proof.
\end{proof}

\begin{lemma} \label{lemma beta_f approx delta_sigma}
    Let $f$ be a strongly $q$-additive function and suppose $\beta_f = \alpha_f$, then 
    \begin{equation} \label{eq beta_f aprox delta_sigma}
        \beta_f = \lim_{M \to \infty} \left( \sup_{t \geq 1} d( \{ n \in \N \colon f(n+t) = f(n) \pmod M \}) \right) .
    \end{equation}
\end{lemma}

\begin{proof}
    Using \cref{lemma alpha_f mod M} and the notation in its proof, if we denote $\beta_f^M = \sup_{t \geq 1} d( \{ n \in \N \colon f(n+t) = f(n) \pmod M \})$ we get that 
     \begin{equation*}
         \beta_f \leq \liminf_{M \to \infty} \beta_f^M \leq \limsup _{M \to \infty} \beta_f^M  \leq \limsup_{M \to \infty} \alpha_f^M = \alpha_f = \beta_f. \qedhere
     \end{equation*}
\end{proof}

With this we are ready to prove \cref{thrm density partial rigidity rates}.

\begin{proof}[Proof of \cref{thrm density partial rigidity rates}]
    Consider $\cF$ the family of functions from \cref{prop family of ex}. The set 
    \begin{equation*}
        B_\cF = \{ \beta_f \colon f \in \cF \} = \{ k/q \in [0,1] \colon q \geq 5, 2 \leq k \leq q-3\}
    \end{equation*}
    is dense in $[0,1]$ (actually $B_\cF = (0,1) \cap \Q$). By \cref{prop family of ex}, every function $f \in \cF$ satisfies the assumptions of \cref{lemma beta_f approx delta_sigma} and so using the notation from \cref{sec preliminaries} together with \cref{thrm partial rigidity main}, equation \eqref{eq beta_f aprox delta_sigma} can be restated as
    \begin{equation*}
         \beta_f = \lim_{M \to \infty} \delta_{\sigma_{f,M}}.
    \end{equation*}
    In particular if $\Delta_\cF = \{ \delta_{\sigma_{f,M}} \in [0,1] \colon f \in \cF, M \geq 2\}$, then $B_\cF \subset \overline{\Delta_\cF} $ and therefore $\Delta_\cF$ is dense in $[0,1]$ concluding the theorem. 
\end{proof}

\appendix

\section{Model of $q$-additive function} \label{appendix}

The definition of substitution is restricted to a finite alphabet $\cA$ to study dynamics in the compact metric space $\cA^{\Z}$, but as an algebraic object it is always possible to define a substitution $\sigma \colon \cA^* \to \cA^*$ as a morphism for concatenation without any assumption on the cardinality of $\cA$. In this appendix we show a result considered folklore that any strongly $q$-additive function can be derived from a substitution sequence on the countable alphabet $\Z$ (see \cref{prop subs strongly q-additive}). For general $q$-additive functions, we use a sequence of such substitutions (see \cref{prop subs general q-additive}). Taking these sequences $\pmod M$, one recovers the substitutions $\sigma_{f,M}$ and directive sequences $(\sigma_{n,f,M})_{n \in \N}$ on the corresponding subsets of $\{0, \ldots, M-1\} = \Z/M$ introduced in \cref{sec preliminaries}. 

We denote by $\Z^*$ the finite words with letters in $\Z$ and to distinguish between multiplication and concatenation we denote the concatenation of $u,v \in \Z^*$ by $u\smwhtsquare v$. 
%For all $r \geq 1$, we also write for $a,b_1,\ldots, b_r \in \Z$, $ a \oplus (b_1 \smwhtsquare \cdots \smwhtsquare b_r) =  (a  + b_1) \smwhtsquare \cdots \smwhtsquare ( a + b_r)$. 

\begin{proposition} \label{prop subs strongly q-additive}
    Let $f \colon \N \to \N$ be a strongly $q$-additive function. Let $\sigma \colon \Z^* \to \Z^*$ be the substitution given by $\sigma(n) = n \smwhtsquare (n+f(1)) \smwhtsquare \cdots \smwhtsquare (n+f(q-1))$ for all $n \in \Z$. If $x$ denotes the infinite sequence $x =  \lim_{m \to \infty} \sigma^m(0)$, then for all $k \in \N$, $x_k = f(k)$.
\end{proposition}
%Notice that with the previous notation $\sigma(n) = n \oplus \sigma(0)$ for all $n \in \Z$.
\begin{proof}
    Notice that, by definition,  $\sigma(0)  = f(0) \smwhtsquare f(1) \smwhtsquare \cdots \smwhtsquare f({q-1}), $ 
    so the consequence for $0 \leq k < q$ is trivially true. By induction, suppose it is already proved for all $k \leq m q -1$, then we want to prove it for $mq, mq+1, \ldots, (m+1)q-1$. Now by construction $x_{[mq, (m+1)q)} = \sigma(x_m)$ and by inductive assumption $x_m = f(m)$ then 
    \begin{align*}
        x_{[mq, (m+1)q)} &= f(m) \smwhtsquare (f(m) + f(1)) \smwhtsquare \cdots  \smwhtsquare (f(m) + f({q-1})) \\ 
        &= f(qm) \smwhtsquare f(qm + 1) \smwhtsquare \cdots  \smwhtsquare f(mq + {q-1}),
    \end{align*}
    concluding the proof. 
\end{proof}

For a general $q$-additive function $f$, we need a sequence of substitutions and the proof is slightly more technical.  

\begin{proposition} \label{prop subs general q-additive}
    Let $f \colon \N \to \N$ be a $q$-additive function. For $j \in \N$, let $\sigma_j \colon \Z^* \to \Z^*$ be the substitution given by $\sigma_j(n) = n \smwhtsquare (n+f(q^j)) \smwhtsquare \cdots \smwhtsquare (n+f(q^j(q-1))) $ for all $n \in \Z$. If $x$ denotes the infinite word $x =  \lim_{j \to \infty} ( \sigma_0 \circ \cdots \circ \sigma_j) (0)$, then for all $k \in \N$, $x_k = f(k)$.
\end{proposition}

\begin{proof}
    As before, by definition, 
    $\sigma_0(0) = f(0) \smwhtsquare f(1) \smwhtsquare \cdots \smwhtsquare f({q-1}), $ 
    so the consequence for $0 \leq k < q$ is trivially true. Consider $f_1 (n) = f(q n)$ for all $n$ and notice that $f_1$ is also a $q$-additive function. Using the same reasoning, we get that
    \begin{align*}
        \sigma_1(0) =  f_1(0) \smwhtsquare f_1(1) \smwhtsquare \cdots \smwhtsquare f_1({q-1}) 
        = f(0) \smwhtsquare  f(q) \smwhtsquare \cdots \smwhtsquare  f(({q-1})q)
    \end{align*}
    Thus, 
    \begin{align*}
        (\sigma_0 \circ \sigma_1)(0) =&  \sigma_0[0] \smwhtsquare  \sigma_0[f(q)] \smwhtsquare \sigma_0[f(2q)]  \smwhtsquare \cdots \smwhtsquare  \sigma_0[f(({q-1})q)] 
    \end{align*}
    where for $0 \leq i < q$
    \begin{align*}
        \sigma_0[f(qi)] &= f(iq) \smwhtsquare (f(iq) + f(1) )\smwhtsquare \cdots \smwhtsquare (f(iq) + f(q-1)) \\
        &= f(iq) \smwhtsquare f(iq +1) \smwhtsquare \cdots \smwhtsquare f(iq + q -1)
    \end{align*}
    Therefore, for $0 \leq k < q^2$, $[(\sigma_0 \circ \sigma_1)(0)]_k = f(k)$. Notice that any $q$-additive function is also a $q^2$-additive function for the base. So one can repeat the previous proof to deduce the equality for $k < q^4$ and iterating the same reasoning it is proven for all $k \in \N$. 
\end{proof}

\bibliographystyle{abbrv} 

%\bibliographystyle{alpha}

% We choose the "alpha" reference style
\bibliography{refs} % Entries are in the library.bib file

\end{document}